\documentclass[12pt]{article}
\usepackage[round]{natbib}
\usepackage[latin1]{inputenc}
\usepackage{amsfonts}
\usepackage{amsmath}
\usepackage{amssymb,amscd,latexsym,amsthm}
\usepackage{enumerate}
\usepackage[dvips]{graphicx}
\newtheorem{theorem}{Theorem}

\title{\vspace{-0.in}\parbox{\linewidth }{\footnotesize\noindent
} \\  \bf Central configurations of the spatial seven-body problem: two twisted triangles plus one}

\author{
\\ Allyson Oliveira \\
\small {Departamento de Matem\'atica, Universidade Federal de Sergipe}
\\ \small{S\~ao Crist\'ov\~ao-SE, CEP. 49100-000, Brazil e-mail: allyson@mat.ufs.br }
}

\date{ }

\begin{document}
\maketitle

\author{ \ }

\begin{abstract}
In this work we are interested in the central configurations of the spatial seven-body problem where six of them are at vertices of two congruents equilateral triangles belong to parallel planes and one triangle is a rotation by the angle $\pi/3$ from the other one. We show a existence and uniqueness of a class of central configuration with this shape, namely six bodies with equal masses at the vertice of a regular octahedron with the seventh body in its center. There is not restriction to the mass of the seventh body.
\end{abstract}


\maketitle

\section{Introduction}

The N-body problem concerns the study of the dynamics of $N$ point masses subject to their mutual Newtonian gravitational interaction. If $N>2$, it is impossible to obtain a general solution to this problem because of its non-integrability; hence, particular solutions are very important in this case. Central configurations are the configurations  for which the total Newtonian acceleration of every body is equal to a constant multiplied by the position vector of this body with respect to the center of mass of the configuration. One of the reasons why central configurations are interesting is because they allow us to obtain explicit homographic solutions of the N-body problem i. e. motion where the configuration of the system changes its size but keeps its shape. Moreover, they arise as the limiting configuration of a total collapse. The papers dealing with central configurations have focused on several aspects such as finding examples of particular central configurations, giving the number of central 
configurations and studying their symmetry, stability,
stacking properties, etc. A stacked central configuration is a central configuration that contains others central configurations by dropping some bodies. Some important references on this theme are \cite{albouy1}, \cite{albouy2}, \cite{hampton}, \cite{hampton1}, \cite{hagihara}, \cite{saari} and \cite{wintner}.

This work deals with central configurations of the spatial seven-body problem where six masses are at the vertices of two equilateral triangles. Those triangles belong to parallel planes and one of then is the rotation by the angle $\pi/3$ rad with respect to the other one. The seventh body is located between these planes on the line perpendicular passing through the center of the triangles and equidistant to them. \cite{mello} studied the case where the triangles are not rotated and they proof the existence of one and only one class of central configuration of that problem. Our results are similar to them as they had predicted. The masses at the triangles must be equal and the only central configuration is that in which those masses are at vertices of a regular octahedron and it does not depends on the value of the seventh mass.
\section{Preliminaries}\label{preliminaries}

Consider $N$ masses, $m_1,...,m_N$, in $\mathbb{R}^3$ subject to their mutual Newtonian gravitational interaction. Let $M=diag\{m_1,m_1,...,m_N,m_N\}$ be the matrix of masses and let $r=(r_1,...,r_N), r_i\in \mathbb{R}^3$ be the position vector. The equations of motion in an inertial reference frame with origin at the center of mass are given by
$$M\ddot{r} = \frac{\partial V}{\partial r},$$
where $V(r_1,...,r_N)=\displaystyle\sum_{1\leq i<j\leq N}\frac{m_im_j}{\|r_i-r_j\|}.$

A non-collision configuration $r=(r_1,...,r_N)$ with $\sum_{i=1}^Nm_ir_i =0$ is a {\it central configuration} if there exists a constant $\lambda$ such that
$$M^{-1}V_r = \lambda r.$$

The value of $\lambda$ is uniquely determined by $\lambda = \frac{U}{2I},$ where $I=\frac{1}{2}\sum_{i=1}^{N}m_i|r_i|^2$ is the moment of inertia of the system. So $\lambda$ is a positive constant.

There are some equivalents ways to describe central configurations. Nonplanar central configurations can be given by Laura-Androyer-Dziobeck equations (see \cite{hagihara})

\begin{equation}\label{dziobeck}
f_{ijh}=\sum_{\substack{k=1\\ k \neq i,j,h}}^{n}m_{k}(R_{ik}-R_{jk})\bigtriangleup_{ijhk}=0,
\end{equation}
for $1\leq i < j \leq n, h=1,...,n, h\neq i,j.$ Here $r_{ij}=\|r_i-r_j\|,R_{ij}=r_{ij}^{-3}$ and $\bigtriangleup_{ijhk}=(r_{i}-r_{j})\wedge(r_{j}-r_{h}).(r_{h}-r_{k})$ is six time the signed volume of the tetrahedron formed by the bodies at positions $r_{i}, r_{j}, r_{h}$ and $r_{k}.$ This is a system with $n(n-1)(n-2)/2$ equations very useful in problem with many symmetries.

\begin{figure}
  \centering
  \includegraphics[width=3in]{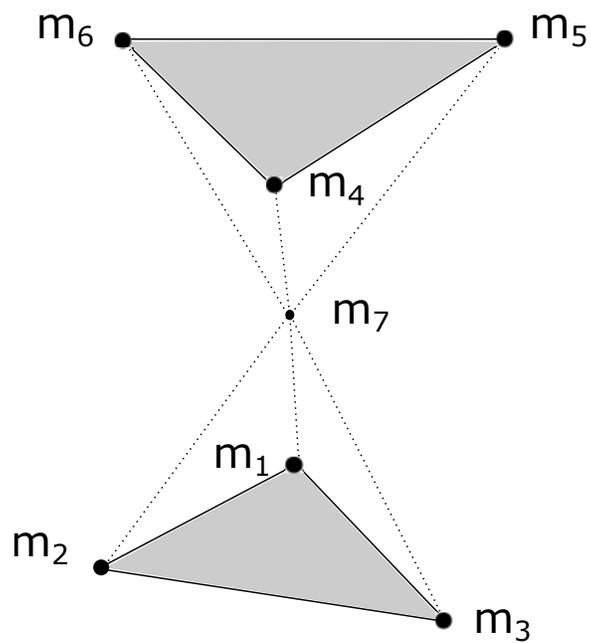}
  \caption{The seven-bodies problem studied here.}
  \label{7bodies}
\end{figure}

\begin{figure}
  \centering
  \includegraphics[width=3in]{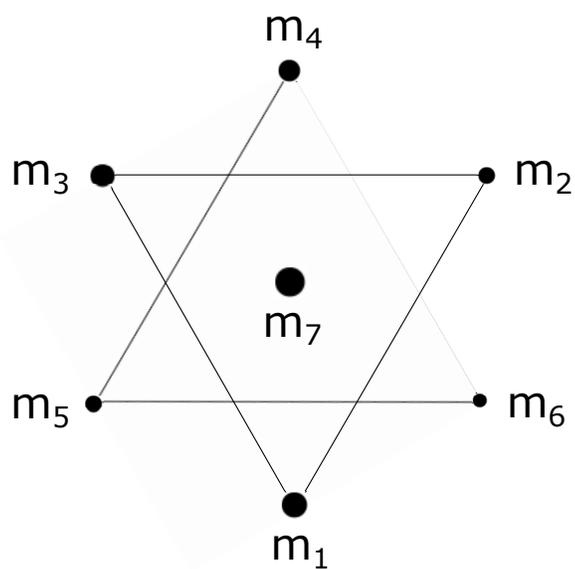}
  \caption{The top view of the seven-bodies problem studied here.}
  \label{topview}
\end{figure}
\section{Main Result}

We are interested in a particular spatial seven bodies problem where six masses are at the vertices of two equilateral triangles belong to parallel planes and one of then is a rotation by the angle $\pi/3$ with respect to the other one. The seventh body is located between these planes on the line perpendicular passing through the center of the triangles and equidistant to them. See Figure \ref{7bodies}. Our main result is the theorem following which characterizes such configurations.
\begin{theorem}\label{lemaequalmasses} 
Consider seven bodies with position vectors $r_1,r_2,r_3,r_4,r_5,r_6$ and $r_7$ and masses $m_1,m_2,m_3,m_4,m_5,m_6$ and $m_7$, respectively. Suppose that $r_1,r_2, r_3$ are at the vertices of a equilateral triangle and $r_4,r_5,r_6$ are at vertices of another equilateral triangle located in a plane parallel to the first one. The second triangle coincides with the first one under a rotation by the angle $\pi/3$ and a projection to its plane. $r_7$ is located between these planes on the line perpendicular passing through the center of the triangles and equidistant to them. See Figure \ref{7bodies} and Figure \ref{topview}. Then we have a central configuration if and only if $m_1=m_2=m_3=m_4=m_5=m_6$, $r_1,r_2,r_3,r_4,r_5$ and $r_6$ are at the vertices of a regular octahedron with $r_7$ at its center. See Figure \ref{octahedron}.
\end{theorem}

\begin{proof}
Using equations \ref{dziobeck}, we have
$$f_{ijh}=\sum_{\substack{k=1\\ k \neq i,j,h}}^{7}m_{k}(R_{ik}-R_{jk})\bigtriangleup_{ijhk}=0,
$$
for $1\leq i < j \leq 7, h=1,...,7, h\neq i,j.$

The symmetry given by our assumption implies in 27 already identically zero equations:
$$f_{124}=f_{125}=f_{127}=f_{134}=f_{136}=f_{137}=f_{147}=f_{174}=f_{235}=f_{236}=f_{237}=f_{257}=f_{275}=
f_{367}=$$ $$=f_{376}=f_{451}=f_{452}=f_{457}=f_{461}=f_{463}=f_{467}=f_{471}=f_{562}=f_{563}=f_{567}=f_{572}=f_{673}=0.$$

Consider the equation
$$f_{123}=0 \Leftrightarrow$$
$$m_{4}(R_{14}-R_{24})\bigtriangleup_{1234}+m_{5}(R_{15}-R_{25})\bigtriangleup_{1235}+m_{6}(R_{16}-R_{26})\bigtriangleup_{1236}+m_{7}(R_{17}-R_{27})\bigtriangleup_{1237}=0.$$
By our hypothesis we get $R_{17}=R_{27}, R_{16}=R_{26}, R_{24}=R_{15}$ and $R_{14}=R_{25}.$ Moreover $\bigtriangleup_{1234}=\bigtriangleup_{1235}.$ So $f_{123}=0$ is equivalent to 
$$(m_{4}-m_{5})(R_{14}-R_{24})\bigtriangleup_{1234}=0,$$
which implies $m_{4}=m_{5}$ since $R_{14}\neq R_{24}$ and $\bigtriangleup_{1234}\neq 0.$

Analogously the equations $f_{135}=0, f_{456}=0$ and  $f_{462}=0$ imply $m_{6}=m_{4}, m_{2}=m_{1}$ and $ m_{1}=m_{3},$  respectively.
 Now taking into account $\bigtriangleup_{1426}=-\bigtriangleup_{1423}$ and $\bigtriangleup_{1425}=0$ the equation $f_{142}=0$ implies
 $$(R_{13}-R_{34})(m_3+m_6)\bigtriangleup_{1423}=0$$ and so 
 $$R_{13}=R_{34},$$
 since $\bigtriangleup_{1423}\neq 0.$
 This in turn implies the equalities
 \begin{equation}\label{eqdistances}
R_{13}=R_{12}=R_{23}=R_{45}=R_{56}=R_{46}=R_{34}=R_{35}=R_{15}=R_{16}=R_{26}=R_{24}.\end{equation}
Therefore the six bodies $m_1,...,m_6$ form a regular octahedron.

Now we look to equation $$0=f_{162}=m_{3}(R_{13}-R_{63})\bigtriangleup_{1623}+m_{4}(R_{14}-R_{64})\bigtriangleup_{1624}+m_{5}(R_{15}-R_{65})\bigtriangleup_{1625}+m_{7}(R_{17}-R_{67})\bigtriangleup_{1627}.$$
Since $\bigtriangleup_{1623}=\bigtriangleup_{1624}, R_{17}=R_{67}$ and $R_{15}=R_{65}$ we get
$$(R_{13}-R_{63})\bigtriangleup_{1623}(m_3-m_4)=0$$
and so $m_3=m_4.$
Therefore all masses at the triangles have the same value, say $m_1=m_2=m_3=m_4=m_5=m_6=m.$  

Assuming equal masses at the triangles the following 12 equations are already satisfied: $$f_{126}=0, f_{135}=0,f_{234}=0,f_{453}=0,f_{462}=0,f_{561}=0,f_{123}=0,f_{456}=0,f_{132}=0,f_{465}=0,$$
$$f_{231}=0 \mbox{ and } f_{564}=0.$$
Assuming  $m_1=m_2=m_3=m_4=m_5=m_6$  we also get :
$$f_{142}=f_{146}=-f_{143}=-f_{145}=f_{253}=-f_{251}=f_{254}=-f_{256}=f_{361}=-f_{362}=$$
$$-f_{364}=f_{365}=-f_{164}=-f_{245}=-f_{356}=-f_{163}=-f_{241}=-f_{352}=f_{152}=f_{263}=$$
$$f_{265}=f_{341}=f_{346}=f_{154}=-2f_{162}=2f_{153}=-2f_{243}=2f_{261}=2f_{342}=-2f_{351}=$$
$$-2f_{165}=-2f_{246}=-2f_{254}=2f_{156}=2f_{264}=2f_{345}=0.$$
If we suppose (\ref{eqdistances}) these 36 equations are already identically zero.

Finally, we will take a look in those 45 equations $f_{ijh}=0$ where $j=7$ or $k=7$ with $i<j.$
First remember that we have already satisfied the following equations:
$$f_{127}=0, f_{237}=0, f_{457}=0, f_{567}=0, f_{137}=0, f_{467}=0, f_{147}=0,$$
$$f_{257}=0, f_{367}=0, f_{174}=0, f_{673}=0, f_{275}=0, f_{471}=0, f_{376}=0, f_{572}=0.$$

Moreover, for the remaining ones we have
$$f_{167}=f_{247}=f_{357}=-f_{157}=-f_{347}=-f_{267}=-2f_{172}=-2f_{273}=$$
$$-2f_{475}=-2f_{576}=-2f_{371}=-2f_{674}=2f_{173}=2f_{271}=2f_{372}=$$
$$2f_{476}=2f_{574}=2f_{675}=2f_{175}=2f_{671}=2f_{276}=2f_{472}=2f_{374}=$$
$$2f_{573}=2f_{176}=2f_{672}=2f_{274}=2f_{473}=2f_{375}=2f_{571}.$$

Since $\Delta_{1673}=\Delta_{1674}=0$ and $\Delta_{1672}=-\Delta_{1675}$, the equation $f_{167}=0$ becomes
$$2m(R_{12}-R_{26})\Delta_{1672}=0,$$
which is satisfied if and only if (\ref{eqdistances}) occurs.

Therefore a central configuration of the seven-body problem with shape given by our hypothesis there exists if and only if $m_1=m_2=m_3=m_4=m_5=m_6$ and $(\ref{eqdistances})$ holds true.

\end{proof}

\begin{figure}
  \centering
  \includegraphics[width=3in]{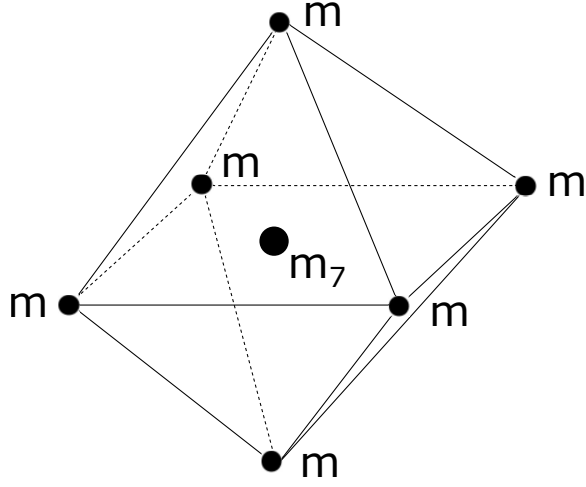}
  \caption{The regular octahedron: The only possible central configuration of the seven-bodies problem studied here.}
  \label{octahedron}
\end{figure}

\section{Remarks}

Our result is in accordance with the provisions by \cite{mello}. In fact, we obtain basically the same results but the calculation here were a little easier. If we take a system of coordinates in which
$$r_1=(l,0,-d),r_2=(-l/2,\sqrt{3}l/2,-d),r_3=(-l/2,-\sqrt{3}l/2,-d),$$
$$r_4=(-l,0,d),r_5=(l/2,-\sqrt{3}l/2,d),r_6=(l/2,\sqrt{3}l/2,d)$$
and $r_7=(0,0,0).$ 
The equation (\ref{eqdistances}) becomes
$$R_{13}=R_{34} \Leftrightarrow r_{13}=r_{34}\Leftrightarrow l=\sqrt{2}d.$$

Again, as well as the results and provisions by \cite{mello}, there is not restriction on the choice of $m_7.$
Equation (\ref{eqdistances}) says that the masses $m_1,m_2,m_3,m_4,m_5$ and $m_6$ are at the vertices of a regular octahedron, a well known central configuration. So the configuration studied in this paper results in a stacked central configuration with seven bodies. See Figure \ref{octahedron}.

\bibliographystyle{plainnat}

\medskip
\medskip

\end{document}